\documentclass[12pt,reqno]{amsart}
\usepackage[left=3cm,top=2cm,right=3cm,bottom=2cm]{geometry}
\usepackage{amssymb}
\usepackage{amsbsy}
\usepackage{tikz}
\usepackage{float}
\usepackage[mathscr]{euscript}

\AtBeginDocument{%
   \def\MR#1{}
}

\begin{document}
\pagestyle{myheadings}

\title{Eigenvalue location in cographs}
\keywords{cograph, adjacency matrix, eigenvalue}
\subjclass{05C50, 05C85, 15A18}

\author{David P. Jacobs}
\address{ School of Computing, Clemson University Clemson, SC 29634 USA}
\email{\tt dpj@clemson.edu}
\author{Vilmar Trevisan}
\address{Instituto de Matem\'atica, UFRGS,  91509--900 Porto Alegre, RS, Brazil}
\email{\tt trevisan@mat.ufrgs.br}
\author{Fernando Colman Tura}
\address{Departamento de Matem\'atica, UFSM,  97105--900 Santa Maria, RS, Brazil}
\email{\tt ftura@smail.ufsm.br}
\pdfpagewidth 8.5 in
\pdfpageheight 11 in

\def\floor#1{\left\lfloor{#1}\right\rfloor}
\newcommand{\Gofr}{G_r}
\newcommand{\Goneclass}{\mathscr G}
\newcommand{\Real}{{\mathbb R}}
\newcommand{\union}{\cup}
\newcommand{\compunion}{\overline{\union}}
\newcommand{\join}{\otimes}
\newcommand{\dom}[1]{\gamma(#1)}
\newcommand{\theratio}[2]{ \lceil \frac{#1}{#2} \rceil }
\newcommand{\avgdeg}[1]{2 - \frac{2}{#1} }
\newcommand{\floorratio}[2]{ \lfloor \frac{#1}{#2} \rfloor }
\newcommand{\Prf}{{\bf Proof: }}
\newcommand{\PrfSketch}{{\bf Proof (Sketch): }}
\newcommand{\boldc}{\mbox{\bf c}}
\newcommand{\sign}{\mbox{sign}}
\newtheorem{Thr}{Theorem}
\newtheorem{Pro}{Proposition}
\newtheorem{Con}{Conjecture}
\newtheorem{Cor}{Corollary}
\newtheorem{Lem}{Lemma}
\newtheorem{Obs}{Observation}
\newtheorem{Rem}{Remark}
\newtheorem{Clm}{Claim}
\newtheorem{Que}{Question}
\newtheorem{Fac}{Fact}
\newtheorem{Ex}{Example}
\newtheorem{Def}{Definition}
\newtheorem{Prop}{Proposition}
\def\floor#1{\left\lfloor{#1}\right\rfloor}

\newenvironment{my_enumerate}{
\begin{enumerate}
  \setlength{\baselineskip}{14pt}
  \setlength{\parskip}{0pt}
  \setlength{\parsep}{0pt}}{\end{enumerate}
}
\newenvironment{my_description}{
\begin{description}
  \setlength{\baselineskip}{14pt}
  \setlength{\parskip}{0pt}
  \setlength{\parsep}{0pt}}{\end{description}
}

\begin{abstract}
We give an $O(n)$ time and space algorithm
for constructing a diagonal matrix congruent to $A + xI$,
where $A$ is the adjacency matrix of a cograph and $x \in \Real$.
Applications include determining the number of eigenvalues of a cograph's adjacency matrix that
lie in any interval,
obtaining a formula for the inertia of a cograph,
and exhibiting infinitely many pairs of equienergetic cographs with integer energy.
\end{abstract}

\thanks{Work supported by Science without Borders CNPq - Grant 400122/2014-6, Brazil}

\maketitle

\section{Introduction}
\label{sec:intro}
Let $G = (V,E)$ be an undirected graph
with vertex set $V$ and edge set $E$.
For $v \in V$, $N(v)$ denotes the {\em open neighborhood} of $v$, that is,
$\{ w | \{v,w\} \in E \}$.
The {\em closed neighborhood} $N[v] = N(v) \cup \{v\}$.
If $|V| = n$, the {\em adjacency matrix} $A=[a_{ij}]$ is the $n \times n$ matrix of zeros and ones
such that $a_{ij} = 1$ if and only if $v_i$ is adjacent to $v_j$
(that is, there is an edge between $v_i$ and $v_j$).
A value $\lambda$ is an {\em eigenvalue} if $\det(A - \lambda I) = 0$,
and since $A$ is real symmetric its eigenvalues are real.
In this paper, a graph's {\em eigenvalues} are the eigenvalues of its adjacency matrix.

This paper is concerned with {\em cographs}.
This class of graphs has been discovered independently by several authors
in many equivalent ways since the 1970's.
Corneil, Lerchs and Burlingham
\cite{Corneil1981}
define cographs recursively:
\begin{enumerate}
\item A graph on a single vertex is a cograph;
\item A finite union of cographs is a cograph;
\item The complement of a cograph is a cograph.
\end{enumerate}
A graph is a cograph if and only
it has no induced path of length four \cite{Corneil1981}.
They are often simply called $P_4$ {\em free} graphs in the literature.
Linear time algorithms for recognizing cographs are given in
\cite{Corneil1985} and more recently in \cite{Habib2005}.

While recognition algorithms for cographs is an interesting problem,
our motivation for considering cographs comes from {\em spectral graph theory}
\cite{Brouwer2012,CDS1995}.
Spectral properties of cographs were studied by Royle in \cite{Royle2003}
where the surprising result was obtained that the rank of a cograph is
the number of non-zero rows in the adjacency matrix.
An elementary proof of this property
was later given in \cite{Chang2008}.
More recently, in \cite{BSS2011} B{\i}y{\i}ko{\u{g}}lu, Simi{\'c} and Stani{\'c}
obtained the multiplicity of $-1$ and $0$ for cographs.

The purpose of this paper is to extend to cographs
eigenvalue location algorithms that exist
for trees \cite{JT2011}, threshold graphs \cite{JTT2013} and generalized lollipop graphs \cite{Del-Vecchio2015}.
Recall that two real symmetric matrices $R$ and $S$ are {\em congruent}
if there exists a nonsingular matrix $P$ for which $R = P^T S P$.
Our main focus is an algorithm that uses $O(n)$ time and space
for constructing a diagonal matrix congruent to $A + xI$,
where $A$ is adjacency matrix of a cograph, and $x \in \Real$.
Our paper is similar in spirit to the papers \cite{JT2011,JTT2013}
which describe $O(n)$ diagonalization algorithms
for trees and for threshold graphs.
Threshold graphs are $P_4$, $C_4$, and $2K_2$ free,
and therefore are a subclass of cographs.
Hence our algorithm is an extension of the algorithm in \cite{JTT2013}.

Several points are worth noting.
First, while one might expect linear time algorithms for graphs with {\em sparse} adjacency
matrices such as trees, the adjacency matrix of a cograph can be {\em dense}.
Next, while our algorithm's correctness is based on elementary matrix operations,
its implementation operates directly on the cotree and uses only $O(n)$ space.
Finally, the {\em analysis} of algorithms for trees and threshold graphs
has led to interesting theoretical results.
For example, in \cite{Patuzzi2014}
conditions were determined for the index (largest eigenvalue) in trees to be integer.
In \cite{JTT2013} the authors showed that all eigenvalues
of threshold graphs, except $-1$ and $0$, are simple.
In \cite{JTT2015} the algorithm was used to show that no threshold graphs
have eigenvalues in $(-1,0)$.

If $G$ is a graph having eigenvalues $\lambda_1, \ldots, \lambda_n$, its {\em energy},
denoted $E(G)$ is defined to be $\sum_{i=1}^n |\lambda_i|$.
Two non-cospectral graphs with the same energy are called
{\em equienergetic}.
Finding non-cospectral equienergetic graphs is a relevant problem.
In \cite{JTT2015} the authors presented infinite sequences of connected,
equienergetic pairs of non-cospectral threshold graphs with integer energy.
In this paper, we continue this investigation.

Here is an outline of the remainder of this paper.
In Section~\ref{sec:cotrees} we describe cotrees, and present some known facts.
In Section~\ref{sec:diagstep} we give the elementary matrix operations used in our algorithm.
In Section~\ref{sec:diagalgo}
we give the complete diagonalization algorithm.
In Section~\ref{sec:locating},
using Sylvester's Law of Inertia,
we show how to efficiently determine how many eigenvalues of
a cograph lie in a given interval.
The {\em inertia} of a graph $G$ is the triple
$(n_{+}, n_{0}, n_{-})$ giving the number of eigenvalues
of $G$ that are positive, zero, and negative,
and in Section~\ref{sec:inertia} we give a formula for cograph inertia.
Finally in Section~\ref{sec:energy}
we exhibit infinitely many non-threshold cographs
equienergetic to a complete graph.

\section{Cotrees and adjacency matrix}
\label{sec:cotrees}
Cographs have been represented in various ways,
and it is useful to recall the representation given
in \cite{Corneil1981}.
The unique {\em normalizend form} of a cograph $G$ is defined recursively:
If $G$ is connected, then it is in normalized form if
it is expressed as a single vertex, or the {\em complemented union} of $k \geq 2$
$$
G = \overline{G_1 \union G_2 \union \ldots \union G_k}
$$
connected cographs $G_i$ in normalized form.
If $G$ disconnected its normalized form
is the complement of a connected cograph in normalized form.
The unique rooted tree $T_G$ representing the parse structure of
the cograph's normalized form is called the {\em cotree}.
The leaves or terminal vertices of $T_G$ correspond to vertices in the cograph.
The interior nodes represent $\compunion$ operations.

It is not difficult to show that the class of cographs is also
the smallest class of graphs containing $K_1$, and closed under
the union $\union$ and join $\join$ operators.
In fact one can transform the cotree
of Corneil, Lerchs and Burlingham
into an equivalent
tree $T_G$ using $\union$ and $\join$.
In the connected case, we simply place a $\join$ at the tree's root,
placing $\union$ on interior nodes with odd depth,
and placing $\join$ on interior nodes with even depth.
To build a cotree for a disconnected cograph, we place $\union$ at the root,
and place $\join$'s at odd depths, and $\union$'s at even depths.
It will be convenient for us to use this unique alternating representation.
In \cite{BSS2011} this structure is called a {\em minimal cotree},
but throughout this paper we call it simply a {\em cotree}.
All interior nodes of cotrees have at least two children.
Figure~\ref{fig:cotree} shows a cograph and cotree.
The following is well known.

\begin{Lem}
\label{lem:lca}
If $G$ is a cograph with cotree $T_G$,
vertices $u$ and $v$ are adjacent in $G$ if and only if
their least common ancestor in $T_G$ is $\join$.
\end{Lem}

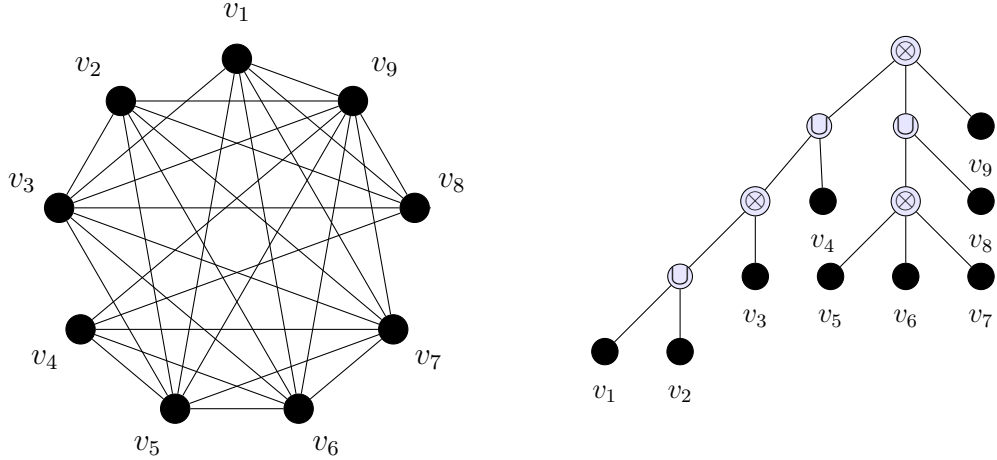
\begin{figure}[h!]
       \begin{minipage}[c]{0.5 \linewidth}
\begin{tikzpicture}
  [scale=0.8,auto=left,every node/.style={circle}]
  \foreach \i/\w in {1/,2/,3/,4/,5/,6/,7/,8/,9/}{
     \node[draw,circle,fill=black,label={360/9 * (\i - 1)+90}: $v_{\i}$ ] (\i) at ({360/9 * (\i - 1)+90}:3) {\w};} 
  \foreach \from in {3,5,6,7,9}{
    \foreach \to in {1,2,...,\from}
      \draw (\from) -- (\to);}
       \foreach \from in {8}{
       \foreach \to in {1,2,3,4,\from}
       \draw (\from) -- (\to)
       ;}
\end{tikzpicture}
       \end{minipage}\hfill
       \begin{minipage}[l]{0.5 \linewidth}
\begin{tikzpicture}
  [scale=1,auto=left,every node/.style={circle,scale=0.9}]
  \node[draw,circle,fill=black,label=below:$v_1$] (p) at (-2,3) {};
  \node[draw,circle,fill=black,label=below:$v_7$] (o) at (3,4) {};
  \node[draw,circle,fill=black,label=below:$v_6$] (n) at (2,4) {};
   \node[draw,circle,fill=black,label=below:$v_5$] (q) at (1,4) {};
  \node[draw, circle, fill=blue!10, inner sep=0] (m) at (2,5) {$\join$};
  \node[draw,circle,fill=black,label=below:$v_8$] (l) at (3,5) {};
  \node[draw,circle,fill=black,label=below:$v_9$] (k) at (3,6) {};
  \node[draw, circle, fill=blue!10, inner sep=0] (j) at (2,6) {$\union$};
  \node[draw,circle,fill=blue!10, inner sep=0] (h) at (2,7) {$\join$};
  \node[draw, circle, fill=blue!10, inner sep=0] (g) at (0.85,6) {$\union$};
  \node[draw,circle,fill=black,label=below:$v_4$] (f) at  (0.9,5) {};
  \node[draw,circle,fill=blue!10, inner sep=0] (a) at (0,5) {$\join$};
  \node[draw, circle, fill=blue!10, inner sep=0] (b) at (-1,4) {$\union$};
  \node[draw,circle,fill=black,label=below:$v_3$] (c) at (0,4) {};
  \node[draw,circle,fill=black,label=below:$v_2$] (d) at (-1,3) {};
  \path (a) edge node[left]{} (b)
        (a) edge node[below]{} (c)
        (b) edge node[left]{} (d)
        (f) edge node[right]{}(g)
        (g) edge node[left]{}(a)
        (h) edge node[right]{}(j)
        (h) edge node[left]{}(g)
        (h) edge node[left]{}(k)
        (j) edge node[right]{}(l)
        (j) edge node[below]{}(m)
        (m) edge node[below]{}(n)
        (m) edge node[right]{}(o)
        (b) edge node[left]{} (p)
        (m) edge node[left]{} (q);
\end{tikzpicture}
       \end{minipage}
       \caption{A cograph $G$ and its cotree.}
       \label{fig:cotree}
\end{figure}

Two vertices $u$ and $v$ are {\em duplicates} if $N(u) = N(v)$
and {\em coduplicates} if $N[u] = N[v]$.
We call $u$ and $v$ {\em siblings} if they are either duplicates or coduplicates.
Siblings play an important role in the structure of cographs,
as well as in this paper.

\begin{Lem}
\label{lem:parentofsiblings}
Two vertices $v$ and $u$ in a cograph are siblings if and only if
they share the same parent $w$ node in the cotree.
Moreover, if $w = \union$, they are duplicates.
If $w = \join$ they are coduplicates.
\end{Lem}
\begin{Lem}
\label{lem:twosiblings}
A cograph $G$ of order $n \geq 2$ has a pair of siblings.
\end{Lem}
\begin{proof}
The cotree of $G$ must have an interior vertex adjacent to two leaves.
\end{proof}
Let $G$ be a cograph with cotree $T_G$.  Let $G - v$ denote the subgraph obtained by removing $v$.
It is known that $G - v$ is a cograph, so we shall use $T - v$ to denote the cotree of $G - v$.
There is a general method for constructing $T - v$ \cite[Lem. 1]{Corneil1981}.
However, it somewhat simplifies the process if $v$ has maximum depth.
The following lemma can be proved with Lemma~\ref{lem:lca}.

\begin{Lem}
\label{lem:cotreeremove}
Let $T_G$ be a cotree, and let $\{v,u\}$ be siblings of greatest depth,
whose parent $w$ has $k$ children.
If $k > 2$ we obtain $T-v$ by removing $v$.
If $k = 2$ and $w$ is not the root,
we obtain $T-v$ by moving $u$ to the parent of $w$, and removing $v$ and $w$.
If $k = 2$ and $w$ is the root, the cotree is $u$.
\end{Lem}

We end this section by making an important observation.
\begin{Lem}
\label{lem:siblingrows}
Let $G$ be a cograph with adjacency matrix $A$ and cotree $T_G$.
Let $\{ v_k, v_l \}$ be siblings with parent $w$.
If $w = \union$ (they are duplicates),
then rows(columns) $k$ and $l$ in $A$ are equal.
If $w = \join$ (they are coduplicates),
then the rows (columns) are equal
except in two positions, namely
$A[k,k] = A[l,l] = 0$ and $A[k,l] = A[l,k] = 1$.
\end{Lem}

\section{Diagonalizing a row and column}
\label{sec:diagstep}
Given a cograph $G$ with adjacency matrix $A$, and $x \in \Real$,
we will transform $B = A + xI$ into a congruent diagonal matrix in stages.
In this section, we illustrate a stage,
given a partially transformed matrix.
Recall that matrices are congruent if one can obtain
the other by a sequence of {\em pairs} of elementary operations,
each pair consisting of a row operation followed by the {\em same} column operation.

Let $T_G$ be the cotree of $G$, and let
$\{ v_k, v_l \}$ be a  pair of siblings.
We assume the diagonal values $d_k$ and $d_l$ of rows $k$ and $l$, respectively,
may have been modified by previous computations,
but {\em only} diagonal values.
The goal is to annihilate off-diagonal $1$'s in the row and column corresponding to $v_k$,
maintaining congruence to $B$.
Let $w$ be the parent of the siblings in $T_G$.
There are two cases.
\paragraph{\bf {Case 1:}} $w = \join$.
By Lemma~\ref{lem:parentofsiblings}
we know $\{v_k, v_l\}$
are coduplicates.
By Lemma~\ref{lem:siblingrows}
rows (columns) $l$ and $k$ of the matrix $B$  have the form
\begin{center}
$ \left[\begin{array}{ccccccccccc}
&            &     &        &    &  a_1   &   &      & a_1    &     & \\
&            &     &        &    & \vdots &   &      & \vdots &     & \\
&            &     &        &    &  a_i   &   &      & a_i    &     & \\
&            &     &        &    & \vdots &   &      & \vdots &     & \\
a_1 & \ldots & a_i & \ldots &    & d_l    &   &      & 1   & \ldots & a_n \\
&   &        &     &        &    &        &   &      &     &\\
a_1 & \ldots & a_i & \ldots &    & 1      &   &      & d_k & \ldots & a_n\\
&   &        &     &        & \vdots      &   &      & \vdots &     & \\
&   &        &     &        & a_n &       &   & a_n  & & \\
\end{array}\right],$
\end{center}
where $a_i\in \{0,1\}$. The row and column operations
\begin{eqnarray*}
R_k \leftarrow R_k - R_{l} \\
C_k \leftarrow C_k - C_{l}
\end{eqnarray*}
give:
\begin{center}
$ \left[\begin{array}{ccccccccccc}
&            &     &        &    &  a_1   &   &      &  0     &     & \\
&            &     &        &    & \vdots &   &      & \vdots &     & \\
&            &     &        &    &  a_i   &   &      &  0     &     & \\
&            &     &        &    & \vdots &   &      & \vdots &     & \\
a_1 & \ldots & a_i & \ldots &    & d_l    &   &      & 1-d_l  & \ldots & a_n \\
&   &        &     &        &    &        &   &      &     &\\
0&\ldots & 0 & \ldots &    & 1-d_l&       &   & d_k+d_l-2 & \ldots & 0\\
&   &        &     &        & \vdots      &   &      & \vdots &     & \\
&   &        &     &        & a_n &       &   & 0  & & \\
\end{array}\right]$ .\\
\end{center}
Most of the non-zero elements in row and column $k$ have been removed.
But we must now remove the two entries $1 - d_l$.
There are three subcases, depending on whether $d_k + d_l -2 \neq 0$ and whether $d_l = 1$.
\subparagraph{\bf subcase 1a:} $d_k + d_l -2 \neq 0$. Then we
may perform the operations
\begin{eqnarray*}
R_{l} \leftarrow R_{l} - \frac{1 - d_l}{d_k + d_l -2} R_{k} \\
C_{l} \leftarrow C_{l} - \frac{1 - d_l}{d_k + d_l -2}
C_{k}
\end{eqnarray*}
obtaining:
\begin{center}
$ \left[\begin{array}{ccccccccccc}
&            &     &        &    &  a_1   &   &      &  0     &     & \\
&            &     &        &    & \vdots &   &      & \vdots &     & \\
&            &     &        &    &  a_i   &   &      &  0     &     & \\
&            &     &        &    & \vdots &   &      & \vdots &     & \\
a_1 & \ldots & a_i & \ldots &    & \gamma &   &      & 0   & \ldots & a_n \\
&   &        &     &        &    &        &   &      &     &\\
0&\ldots & 0 & \ldots &     & 0  &        &   & d_k+d_l-2 & \ldots & 0\\
&   &        &     &        & \vdots      &   &      & \vdots &     & \\
&   &        &     &        & a_n &       &   & 0  & & \\
\end{array}\right]$,\\
\end{center}
where
$$
\gamma = d_l - \frac{ (1 - d_l )^2}{d_k + d_l - 2} =
\frac{d_k d_l - 1}{d_k + d_l - 2} \cdot
$$
The following assignments are made
\begin{equation}
\label{eq:sub1a}
d_k \leftarrow d_k + d_l - 2  \hspace*{1cm}  d_{l} \leftarrow \frac{ d_k d_l -1 }{d_k + d_l - 2} .
\end{equation}
Since row (and column) $k$ is diagonalized,
the value $d_k$ becomes permanent value
and we remove $v_k$ from the cotree:
$$
T_G \leftarrow T_G - v_k .
$$
The assignments in (\ref{eq:sub1a}) are technically incorrect since $d_k$ is modified
in the first assignment and used in the second.
In our algorithm's pseudo-code,
we will use temporary variables and assign $\alpha \leftarrow  d_k$  and $\beta \leftarrow d_{l}$,
and then assign $d_{k} \leftarrow \alpha + \beta -2$
and
$d_{l} \leftarrow \frac{\alpha \beta -1}{\alpha + \beta -2}$.
To keep notation simple,
in the remainder of this section, we will not do this.
\subparagraph{\bf subcase 1b:} $d_k + d_l = 2$ and  $d_l = 1$. Then the matrix looks like
\begin{center}
$ \left[\begin{array}{ccccccccccc}
&            &     &        &    &  a_1   &   &      &  0     &     & \\
&            &     &        &    & \vdots &   &      & \vdots &     & \\
&            &     &        &    &  a_i   &   &      &  0     &     & \\
&            &     &        &    & \vdots &   &      & \vdots &     & \\
a_1 & \ldots & a_i & \ldots &    & 1 &    &    & 0  & \ldots & a_n \\
&   &        &     &        &    &        &   &      &     &\\
0&\ldots & 0 & \ldots &     & 0  &        &   & 0    & \ldots & 0\\
&   &        &     &        & \vdots      &   &      & \vdots &     & \\
&   &        &     &        & a_n &       &   & 0  & & \\
\end{array}\right]$,
\end{center}
and we are done.
We make the assignments
$$
d_k \leftarrow 0 \hspace*{1cm}  d_{l} \leftarrow 1 \hspace*{1cm}  T_G \leftarrow T_G - v_k
$$
as $d_k$ becomes permanent, and $v_k$ is removed.
\subparagraph{\bf subcase 1c:} $d_k+d_l-2= 0$ and $d_l \neq 1$. Then our matrix looks like
\begin{center}
$ \left[\begin{array}{ccccccccccc}
&            &     &        &    &  a_1   &   &      &  0     &     & \\
&            &     &        &    & \vdots &   &      & \vdots &     & \\
&            &     &        &    &  a_i   &   &      &  0     &     & \\
&            &     &        &    & \vdots &   &      & \vdots &     & \\
a_1 & \ldots & a_i & \ldots &    & d_l    &   &      & 1-d_l  & \ldots & a_n \\
&   &        &     &        &    &        &   &      &     &\\
0&\ldots & 0 & \ldots &    & 1-d_l&       &   & 0 & \ldots & 0\\
&   &        &     &        & \vdots      &   &      & \vdots &     & \\
&   &        &     &        & a_n &       &   & 0  & &
\end{array}\right]$ .
\end{center}
We note that $a_1,\ldots, a_n \in \{0,1\}$ and since $1 - d_l \neq 0$, for each $i \neq k,l$
such that $a_{i l}=1,$ we perform:
\begin{eqnarray}
   R_{i} & \leftarrow & R_{i} - \frac{1}{1- d_l} R_{k} \label{eq:subcase1c} \\
   C_{i} &  \leftarrow & C_{i} - \frac{1}{1- d_l} C_{k} \nonumber
\end{eqnarray}
This annihilates most of row (column) $l$, without changing any other values:
\begin{center}
$ \left[\begin{array}{ccccccccccc}
&            &     &        &    &  0     &   &      &  0     &     & \\
&            &     &        &    & \vdots &   &      & \vdots &     & \\
&            &     &        &    & 0      &   &      &  0     &     & \\
&            &     &        &    & \vdots &   &      & \vdots &     & \\
0& \ldots & 0   & \ldots &    & d_l    &   &      & 1-d_l  & \ldots & 0 \\
&   &        &     &        &    &        &   &      &     &\\
0&\ldots & 0 & \ldots &    & 1-d_l&       &   & 0 & \ldots & 0\\
&   &        &     &        & \vdots      &   &      & \vdots &     & \\
&   &        &     &        & 0 &       &   & 0  & &
\end{array}\right]$ .
\end{center}
The operations
\begin{eqnarray*}
R_{l} & \leftarrow & R_{l} + \frac{1}{2} R_{k} \\
C_{l} & \leftarrow & C_{l} + \frac{1}{2} C_{k}
\end{eqnarray*}
replace the diagonal $d_l$ with one, while the operations
\begin{eqnarray*}
R_{k} & \leftarrow & R_{k} - (1- d_l) R_{l} \\
C_{k} &  \leftarrow & C_{k} - (1- d_l) C_{l}
\end{eqnarray*}
eliminate the two off-diagonal elements, finally giving:
\begin{center}
$ \left[\begin{array}{ccccccccccc}
&            &     &        &    &  0     &   &      &  0     &     & \\
&            &     &        &    & \vdots &   &      & \vdots &     & \\
&            &     &        &    & 0      &   &      &  0     &     & \\
&            &     &        &    & \vdots &   &      & \vdots &     & \\
0& \ldots & 0   & \ldots    &    & 1      &   &      & 0  & \ldots & 0 \\
&   &        &     &        &    &        &   &      &     &\\
0&\ldots & 0 & \ldots &    & 0   &       &   & -(1-d_l)^2 & \ldots & 0\\
&   &        &     &        & \vdots      &   &      & \vdots &     & \\
&   &        &     &        & 0 &       &   & 0  & &
\end{array}\right]$ .
\end{center}
What is different about this subcase is that
both rows $k$ and $l$ have been diagonalized.
We make the following assignments
$$
d_k \leftarrow -(1- d_l )^2
\hspace*{1cm}
d_l \leftarrow 1
\hspace*{1cm}
T_G \leftarrow T_G - v_k
\hspace*{1cm}
T_G \leftarrow T_G - v_l
$$
removing both vertices from $T_G$ and making both variables permanent.
\paragraph{\bf {Case 2:}} $w = \union$.
By Lemma~\ref{lem:siblingrows},
row $k$ and $l$ of the matrix $B$ look like:
\begin{center}
$ \left[\begin{array}{ccccccccccc}
&            &     &        &    &  a_1   &   &      & a_1    &     & \\
&            &     &        &    & \vdots &   &      & \vdots &     & \\
&            &     &        &    &  a_i   &   &      & a_i    &     & \\
&            &     &        &    & \vdots &   &      & \vdots &     & \\
a_1 & \ldots & a_i & \ldots &    & d_l    &   &      & 0   & \ldots & a_n \\
&   &        &     &        &    &        &   &      &     &\\
a_1 & \ldots & a_i & \ldots &    & 0      &   &      & d_k & \ldots & a_n\\
&   &        &     &        & \vdots      &   &      & \vdots &     & \\
&   &        &     &        & a_n &       &   & a_n  & & \\
\end{array}\right],$\\
\end{center}
Similar to Case 1, the row and column operations
\begin{eqnarray*}
R_k \leftarrow R_k - R_{l} \\
C_k \leftarrow C_k - C_{l}
\end{eqnarray*}
give:
\begin{center}
$ \left[\begin{array}{ccccccccccc}
&            &     &        &    &  a_1   &   &      & 0    &     & \\
&            &     &        &    & \vdots &   &      & \vdots &     & \\
&            &     &        &    &  a_i   &   &      & 0    &     & \\
&            &     &        &    & \vdots &   &      & \vdots &     & \\
a_1 & \ldots & a_i & \ldots &    & d_l    &   &      & -d_l  & \ldots & a_n \\
&   &        &     &        &    &        &   &      &     &\\
0 & \ldots & 0 & \ldots &    & -d_l     &   &      & d_k+d_l & \ldots & 0\\
&   &        &     &        & \vdots      &   &      & \vdots &     & \\
&   &        &     &        & a_n &       &   & 0  & & \\
\end{array}\right],$\\
\end{center}
\paragraph{\bf subcase 2a:} $d_k + d_l \neq 0.$
Then the matrix operations
\begin{eqnarray*}
R_{l} \leftarrow R_{l} + \frac{d_l}{d_k + d_l}R_{k} \\
C_{l} \leftarrow C_{l} + \frac{d_l}{d_k + d_l}C_{k}
\end{eqnarray*}
diagonalize the matrix and the following assignments are made:
$$
d_{k} \leftarrow  d_k + d_l \hspace*{1cm} T_G \leftarrow T_G - v_k
\hspace*{1cm}
d_{l} \leftarrow \frac{d_k  d_l}{d_k + d_l}
$$
\paragraph{\bf subcase 2b:}
$d_k + d_l=0,$ and $d_l=0.$
Similar to {\bf subcase 1b}, the matrix is in diagonal form and the following assignments are made:
$$
d_{k} \leftarrow  0
\hspace*{1cm}
T_G \leftarrow T_G - v_k
\hspace*{1cm}
d_{l} \leftarrow 0
$$
\paragraph{\bf subcase 2c:} $d_k + d_l=0,$ and $ d_l\neq 0.$
Since $d_l \neq 0$,
we use operations similar to (\ref{eq:subcase1c})
to annihilate most of rows and columns $l$:
\begin{center}
$ \left[\begin{array}{ccccccccccc}
&            &     &        &    &  0     &   &      &  0     &     & \\
&            &     &        &    & \vdots &   &      & \vdots &     & \\
&            &     &        &    & 0      &   &      &  0     &     & \\
&            &     &        &    & \vdots &   &      & \vdots &     & \\
0& \ldots & 0   & \ldots &    & d_l    &   &      & -d_l  & \ldots & 0 \\
&   &        &     &        &    &        &   &      &     &\\
0&\ldots & 0 & \ldots &    & -d_l&       &   & 0 & \ldots & 0\\
&   &        &     &        & \vdots      &   &      & \vdots &     & \\
&   &        &     &        & 0 &       &   & 0  & &
\end{array}\right]$ .
\end{center}
The operations
\begin{eqnarray*}
R_{k} & \leftarrow & R_{k} +  R_{l} \\
C_{k} & \leftarrow & C_{k} +  C_{l}
\end{eqnarray*}
complete the diagonlization.
The following assignments are made:
$$
d_{k} \leftarrow  -d_l  \hspace*{1cm} d_{l} \leftarrow d_l \hspace*{1cm} T_G \leftarrow T_G - v_k \hspace*{1cm} T_G \leftarrow T_G - v_l
$$
\begin{figure}[t]
\input{fig2.tex}
\end{figure}
\section{Diagonalizing $A + xI$}
\label{sec:diagalgo}
The algorithm in Figure~\ref{algo}
constructs a diagonal matrix $D$,
congruent to $B = A + xI$,
where $A$ is the adjacency matrix of a cograph $G$, and $x \in \Real$.
The algorithm's input is the cotree $T_G$ and $x$.
Note the algorithm does not store the matrix, but
rather only records changes on the diagonal values $d_i$ of $B$,
allowing only $O(n)$ space.

It initializes all entries of $D$ with $x$.
At the beginning of each iteration, the cotree represents
the subgraph induced by vertices whose rows and columns have not yet been diagonalized.
During each iteration, a pair of siblings $\{v_k, v_l \}$
from the cotree is selected, whose existence is guaranteed by Lemma~\ref{lem:twosiblings}.
Note that the algorithm can't assume that either of the diagonal elements $d_k$ and $d_l$
are still $x$.

Each iteration of the loop annihilates either one or two rows (columns),
updating diagonal values wth arithmetic from Section~\ref{sec:diagstep}.
Once a row (column) is diagonalized, those entries never participate
again in row and column operations, and so their values remain unchanged.
When a row (column) corresponding to vertex $v$ has been diagonalized,
the subgraph induced by those vertices whose rows (columns) are undiagonalized,
has been reduced.
Thus $v$ is removed from the cotree.
It is slightly easier to reconstruct the cotree
when leaves of maximum depth are removed,
using the method of Lemma~\ref{lem:cotreeremove}.
Hence we always choose sibling pairs of maximum depth.
The algorithm terminates when all rows and columns have been diagonalized.
Each iteration of \verb+Diagonalize+
takes constant time, so its running time is $O(n)$.

\begin{Thr}
\label{main}
For inputs $T_G$ and $x$, where $T_G$ is the cotree of cograph $G$ having adjacency matrix $A$,
algorithm
\verb+Diagonalize+
computes a diagonal matrix $D$, which is congruent to $A + x I$
using $O(n)$ time and space.
\end{Thr}
It is interesting that for threshold graphs, the cotree is a {\em caterpillar},
as shown in \cite[Cor. 3.2]{Sciriha2011}.
When \verb+Diagonalize+
is given such a cotree, it performs essentially the same arithmetic as
the threshold graph algorithm in \cite{JTT2013}.

\begin{Ex}
In the remainder of this section we apply Algorithm \verb+Diagonalize+
to the cograph in Figure~\ref{fig:cotree} with $x=0$.
In our figures, diagonal values $d_i$ appear under the vertex $v_i$ in the cotree.
Initially, all $d_i$ will be $x = 0$.
Sibling pairs appear in red.
Dashed edges indicate vertices about to be removed from the cotree.
To follow this example, the reader need not be concerned with labels of vertices,
but simply check that each cotree produces the next one.
\end{Ex}
In the first iteration of the algorithm,
siblings $\{v_k,v_l\}$ are chosen of maximum depth.
Since $d_k = \alpha = 0$ and $d_l = \beta = 0$,
{\bf subcase~2b}  occurs.
\begin{figure}[h]
       \begin{minipage}[c]{0.45 \linewidth}
\begin{tikzpicture}
  [scale=1,auto=left,every node/.style={circle,scale=0.9}]
  \node[draw,circle,fill=red,label=below:$0$] (p) at (-2,3) {};
  \node[draw,circle,fill=black,label=below:$0$] (o) at (3,4) {};
  \node[draw,circle,fill=black,label=below:$0$] (n) at (2,4) {};
   \node[draw,circle,fill=black,label=below:$0$] (q) at (1,4) {};
  \node[draw, circle, fill=blue!10, inner sep=0] (m) at (2,5) {$\join$};
  \node[draw,circle,fill=black,label=below:$0$] (l) at (3,5) {};
  \node[draw,circle,fill=black,label=below:$0$] (k) at (3,6) {};
  \node[draw, circle, fill=blue!10, inner sep=0] (j) at (2,6) {$\union$};
  \node[draw,circle,fill=blue!10, inner sep=0] (h) at (2,7) {$\join$};
  \node[draw, circle, fill=blue!10, inner sep=0] (g) at (0.85,6) {$\union$};
  \node[draw,circle,fill=black,label=below:$0$] (f) at  (0.9,5) {};
  \node[draw,circle,fill=blue!10, inner sep=0] (a) at (0,5) {$\join$};
  \node[draw, circle, fill=blue!10, inner sep=0] (b) at (-1,4) {$\union$};
  \node[draw,circle,fill=black,label=below:$0$] (c) at (0,4) {};
  \node[draw,circle,fill=red,label=below:$0$] (d) at (-1,3) {};
  \path (a) edge node[left]{} (b)
        (a) edge node[below]{} (c)
        (b) edge node[left]{} (d)
        (f) edge node[right]{}(g)
        (g) edge node[left]{}(a)
        (h) edge node[right]{}(j)
        (h) edge node[left]{}(g)
        (h) edge node[left]{}(k)
        (j) edge node[right]{}(l)
        (j) edge node[below]{}(m)
        (m) edge node[below]{}(n)
        (m) edge node[right]{}(o)
        (b) edge [dashed]node[left]{} (p)
        (m) edge node[left]{} (q);
\end{tikzpicture}
\caption{}
       \label{fig3}
       \end{minipage}\hfill
       \begin{minipage}[c]{0.45 \linewidth}
\begin{tikzpicture}
  [scale=1,auto=left,every node/.style={circle,scale=0.9}]
  \node[draw,circle,fill=black,label=below:$0$] (p) at (-2,3) {};
  \node[draw,circle,fill=black,label=below:$0$] (o) at (3,4) {};
  \node[draw,circle,fill=black,label=below:$0$] (n) at (2,4) {};
   \node[draw,circle,fill=black,label=below:$0$] (q) at (1,4) {};
  \node[draw, circle, fill=blue!10, inner sep=0] (m) at (2,5) {$\join$};
  \node[draw,circle,fill=black,label=below:$0$] (l) at (3,5) {};
  \node[draw,circle,fill=black,label=below:$0$] (k) at (3,6) {};
  \node[draw, circle, fill=blue!10, inner sep=0] (j) at (2,6) {$\union$};
  \node[draw,circle,fill=blue!10, inner sep=0] (h) at (2,7) {$\join$};
  \node[draw, circle, fill=blue!10, inner sep=0] (g) at (0.85,6) {$\union$};
  \node[draw,circle,fill=black,label=below:$0$] (f) at  (0.9,5) {};
  \node[draw,circle,fill=blue!10, inner sep=0] (a) at (0,5) {$\join$};
  \node[draw, circle, fill=black, label=below:$0$] (b) at (-1,4) {};
  \node[draw,circle,fill=black,label=below:$0$] (c) at (0,4) {};
  
  \path (a) edge node[left]{} (b)
        (a) edge node[below]{} (c)
        
        (f) edge node[right]{}(g)
        (g) edge node[left]{}(a)
        (h) edge node[right]{}(j)
        (h) edge node[left]{}(g)
        (h) edge node[left]{}(k)
        (j) edge node[right]{}(l)
        (j) edge node[below]{}(m)
        (m) edge node[below]{}(n)
        (m) edge node[right]{}(o)
        
        (m) edge node[left]{} (q);
\end{tikzpicture}
       \caption{}
       \label{fig4}
        \end{minipage}
\end{figure}
The assignments
$$
d_k   \leftarrow 0  \hspace*{1cm} d_l   \leftarrow  0
$$
are made, represented in Figure~\ref{fig3}.
Figure~\ref{fig4} depicts the cotree after vertex $v_k$ is removed
and $v_l$ is relocated under its parent's parent using the
rules in Lemma~\ref{lem:cotreeremove}.

In the next step,
a sibling pair $\{v_k, v_l\}$ of depth three is chosen.
Since their parent is $\join$,
and $d_k = \alpha=0$ and $d_l=\beta=0$,
{\bf subcase~1a}
is executed and the following
assignments are made:
$$
d_k    \leftarrow  -2 \hspace*{1cm} d_l   \leftarrow  \frac{1}{2}
$$
Figure~\ref{fig5}
shows the cotree after these assignments,
and Figure \ref{fig6} shows the cotree with $v_k$
removed and $v_l$ relocated.
%
 \begin{figure}[H]
        \begin{minipage}[c]{0.45 \linewidth}
\begin{tikzpicture}
  [scale=1,auto=left,every node/.style={circle,scale=0.9}]
  \node[draw,circle,fill=black,label=below:$0$] (p) at (-2,3) {};
  \node[draw,circle,fill=black,label=below:$0$] (o) at (3,4) {};
  \node[draw,circle,fill=black,label=below:$0$] (n) at (2,4) {};
   \node[draw,circle,fill=black,label=below:$0$] (q) at (1,4) {};
  \node[draw, circle, fill=blue!10, inner sep=0] (m) at (2,5) {$\join$};
  \node[draw,circle,fill=black,label=below:$0$] (l) at (3,5) {};
  \node[draw,circle,fill=black,label=below:$0$] (k) at (3,6) {};
  \node[draw, circle, fill=blue!10, inner sep=0] (j) at (2,6) {$\union$};
  \node[draw,circle,fill=blue!10, inner sep=0] (h) at (2,7) {$\join$};
  \node[draw, circle, fill=blue!10, inner sep=0] (g) at (0.85,6) {$\union$};
  \node[draw,circle,fill=black,label=below:$0$] (f) at  (0.9,5) {};
  \node[draw,circle,fill=blue!10, inner sep=0] (a) at (0,5) {$\join$};
  \node[draw, circle, fill=red, label=below:$-2$] (b) at (-1,4) {};
  \node[draw,circle,fill=red,label=below:$\frac{1}{2}$] (c) at (0,4) {};
  
  \path (a) edge [dashed]node[left]{} (b)
        (a) edge node[below]{} (c)
        
        (f) edge node[right]{}(g)
        (g) edge node[left]{}(a)
        (h) edge node[right]{}(j)
        (h) edge node[left]{}(g)
        (h) edge node[left]{}(k)
        (j) edge node[right]{}(l)
        (j) edge node[below]{}(m)
        (m) edge node[below]{}(n)
        (m) edge node[right]{}(o)
        
        (m) edge node[left]{} (q);
\end{tikzpicture}
\caption{}
       \label{fig5}
       \end{minipage}\hfill
       \begin{minipage}[c]{0.45 \linewidth}
\begin{tikzpicture}
  [scale=1,auto=left,every node/.style={circle,scale=0.9}]
  \node[draw,circle,fill=black,label=below:$0$] (p) at (-2,3) {};
  \node[draw,circle,fill=black,label=below:$0$] (o) at (3,4) {};
  \node[draw,circle,fill=black,label=below:$0$] (n) at (2,4) {};
   \node[draw,circle,fill=black,label=below:$0$] (q) at (1,4) {};
  \node[draw, circle, fill=blue!10, inner sep=0] (m) at (2,5) {$\join$};
  \node[draw,circle,fill=black,label=below:$0$] (l) at (3,5) {};
  \node[draw,circle,fill=black,label=below:$0$] (k) at (3,6) {};
  \node[draw, circle, fill=blue!10, inner sep=0] (j) at (2,6) {$\union$};
  \node[draw,circle,fill=blue!10, inner sep=0] (h) at (2,7) {$\join$};
  \node[draw, circle, fill=blue!10, inner sep=0] (g) at (0.85,6) {$\union$};
  \node[draw,circle,fill=black,label=below:$0$] (f) at  (0.9,5) {};
  \node[draw,circle,fill=black, label=below:$\frac{1}{2}$] (a) at (0,5) {};
  \node[draw, circle, fill=black, label=below:$-2$] (b) at (-1,3) {};
  \path 
        (f) edge node[right]{}(g)
        (g) edge node[left]{}(a)
        (h) edge node[right]{}(j)
        (h) edge node[left]{}(g)
        (h) edge node[left]{}(k)
        (j) edge node[right]{}(l)
        (j) edge node[below]{}(m)
        (m) edge node[below]{}(n)
        (m) edge node[right]{}(o)
        
        (m) edge node[left]{} (q);
\end{tikzpicture}
       \caption{}
       \label{fig6}
        \end{minipage}
 \end{figure}
Next, another depth three
sibling pair $\{v_k, v_l \}$ is chosen.
Since $d_k = \alpha=0$ and $d_l=\beta=0$, {\bf subcase~1a} is taken
and assignments
$$
d_k    \leftarrow  -2  \hspace*{1cm} d_l   \leftarrow  \frac{1}{2}
$$
made.
Vertex $v_k$ is removed from the cotree
as shown in Figure \ref{fig7}, making
the value $d_k= -2$ permanent.
Note $v_l$ does not move per the
rules in Lemma~\ref{lem:cotreeremove}.
%
 \begin{figure}[h!]
        \begin{minipage}[c]{0.45 \linewidth}
\begin{tikzpicture}
  [scale=1,auto=left,every node/.style={circle,scale=0.9}]
  \node[draw,circle,fill=black,label=below:$0$] (p) at (-2,3) {};
  \node[draw,circle,fill=black,label=below:$0$] (o) at (3,4) {};
  \node[draw,circle,fill=red,label=below:$\frac{1}{2}$] (n) at (2,4) {};
   \node[draw,circle,fill=red,label=below:$-2$] (q) at (1,4) {};
  \node[draw, circle, fill=blue!10, inner sep=0] (m) at (2,5) {$\join$};
  \node[draw,circle,fill=black,label=below:$0$] (l) at (3,5) {};
  \node[draw,circle,fill=black,label=below:$0$] (k) at (3,6) {};
  \node[draw, circle, fill=blue!10, inner sep=0] (j) at (2,6) {$\union$};
  \node[draw,circle,fill=blue!10, inner sep=0] (h) at (2,7) {$\join$};
  \node[draw, circle, fill=blue!10, inner sep=0] (g) at (0.85,6) {$\union$};
  \node[draw,circle,fill=black,label=below:$0$] (f) at  (0.9,5) {};
  \node[draw,circle,fill=black, label=below:$\frac{1}{2}$] (a) at (0,5) {};
  \node[draw, circle, fill=black, label=below:$-2$] (b) at (-1,3) {};
  \path 
        (f) edge node[right]{}(g)
        (g) edge node[left]{}(a)
        (h) edge node[right]{}(j)
        (h) edge node[left]{}(g)
        (h) edge node[left]{}(k)
        (j) edge node[right]{}(l)
        (j) edge node[below]{}(m)
        (m) edge node[below]{}(n)
        (m) edge node[right]{}(o)
        
        (m) edge [dashed]node[left]{} (q);
\end{tikzpicture}
\caption{}
       \label{fig7}
       \end{minipage}\hfill
       \begin{minipage}[c]{0.45 \linewidth}
\begin{tikzpicture}
  [scale=1,auto=left,every node/.style={circle,scale=0.9}]
  \node[draw,circle,fill=black,label=below:$0$] (p) at (-2,3) {};
  \node[draw,circle,fill=red,label=below:$\frac{2}{3}$] (o) at (3,4) {};
  \node[draw,circle,fill=red,label=below:$-\frac{3}{2}$] (n) at (2,4) {};
   \node[draw,circle,fill=black,label=below:$-2$] (q) at (0,3) {};
  \node[draw, circle, fill=blue!10, inner sep=0] (m) at (2,5) {$\join$};
  \node[draw,circle,fill=black,label=below:$0$] (l) at (3,5) {};
  \node[draw,circle,fill=black,label=below:$0$] (k) at (3,6) {};
  \node[draw, circle, fill=blue!10, inner sep=0] (j) at (2,6) {$\union$};
  \node[draw,circle,fill=blue!10, inner sep=0] (h) at (2,7) {$\join$};
  \node[draw, circle, fill=blue!10, inner sep=0] (g) at (0.85,6) {$\union$};
  \node[draw,circle,fill=black,label=below:$0$] (f) at  (0.9,5) {};
  \node[draw,circle,fill=black, label=below:$\frac{1}{2}$] (a) at (0,5) {};
  \node[draw, circle, fill=black, label=below:$-2$] (b) at (-1,3) {};
  \path 
        (f) edge node[right]{}(g)
        (g) edge node[left]{}(a)
        (h) edge node[right]{}(j)
        (h) edge node[left]{}(g)
        (h) edge node[left]{}(k)
        (j) edge node[right]{}(l)
        (j) edge node[below]{}(m)
        (m) edge [dashed]node[below]{}(n)
        (m) edge node[right]{}(o);
\end{tikzpicture}
       \caption{ }
       \label{fig8}
        \end{minipage}
 \end{figure}

Depth three siblings $\{v_k, v_l \}$ are selected again.
Since $d_k=\alpha = \frac{1}{2}$ and $d_l=\beta = 0$,
{\bf subcase~1a} is again executed and the assignments
$$
d_k    \leftarrow  -\frac{3}{2} \hspace*{1cm} d_l   \leftarrow  \frac{2}{3}
$$
are made as shown in Figure \ref{fig8}.
Vertex $v_k$ will be removed from the cotree,
and $v_l$ relocated to its parent's parent, as Figure \ref{fig9} shows.

Next, a depth two pair $\{v_k, v_l\}$ is chosen.
Since $ d_k= \alpha = \frac{1}{2},$ and $d_l= \beta = 0$,
{\bf subcase~2a} is applied,
and the following assignments are made:
$$
d_k \leftarrow  \frac{1}{2} \hspace*{1cm}  d_l   \leftarrow  0
$$
Their values don't change, as indicated in Figure \ref{fig10}.
Vertex $v_k$ is removed from the cotree,
and vertex $v_l$ is relocated
as shown in Figure \ref{fig11}.
Next, a depth two pair $\{v_k, v_l\}$ is selected.
Since $ d_k= \alpha = \frac{2}{3}$ and $d_l= \beta = 0.$
the {\bf subcase~2a} is applied again.
%
 \begin{figure}[H]
        \begin{minipage}[c]{0.45 \linewidth}
\begin{tikzpicture}
  [scale=1,auto=left,every node/.style={circle,scale=0.9}]
  \node[draw,circle,fill=black,label=below:$0$] (p) at (-2,3) {};
  
  \node[draw,circle,fill=black,label=below:$-\frac{3}{2}$] (n) at (1,3) {};
   \node[draw,circle,fill=black,label=below:$-2$] (q) at (0,3) {};
  \node[draw, circle, fill=black, label=below:$\frac{2}{3}$] (m) at (2,5) {};
  \node[draw,circle,fill=black,label=below:$0$] (l) at (3,5) {};
  \node[draw,circle,fill=black,label=below:$0$] (k) at (3,6) {};
  \node[draw, circle, fill=blue!10, inner sep=0] (j) at (2,6) {$\union$};
  \node[draw,circle,fill=blue!10, inner sep=0] (h) at (2,7) {$\join$};
  \node[draw, circle, fill=blue!10, inner sep=0] (g) at (0.85,6) {$\union$};
  \node[draw,circle,fill=black,label=below:$0$] (f) at  (0.9,5) {};
  \node[draw,circle,fill=black, label=below:$\frac{1}{2}$] (a) at (0,5) {};
  \node[draw, circle, fill=black, label=below:$-2$] (b) at (-1,3) {};
  \path 
        (f) edge node[right]{}(g)
        (g) edge node[left]{}(a)
        (h) edge node[right]{}(j)
        (h) edge node[left]{}(g)
        (h) edge node[left]{}(k)
        (j) edge node[right]{}(l)
        (j) edge node[below]{}(m);
\end{tikzpicture}
\caption{}
       \label{fig9}
       \end{minipage}\hfill
       \begin{minipage}[c]{0.45 \linewidth}
\begin{tikzpicture}
  [scale=1,auto=left,every node/.style={circle,scale=0.9}]
  \node[draw,circle,fill=black,label=below:$0$] (p) at (-2,3) {};
  
  \node[draw,circle,fill=black,label=below:$-\frac{3}{2}$] (n) at (1,3) {};
   \node[draw,circle,fill=black,label=below:$-2$] (q) at (0,3) {};
  \node[draw, circle, fill=black, label=below:$\frac{2}{3}$] (m) at (2,5) {};
  \node[draw,circle,fill=black,label=below:$0$] (l) at (3,5) {};
  \node[draw,circle,fill=black,label=below:$0$] (k) at (3,6) {};
  \node[draw, circle, fill=blue!10, inner sep=0] (j) at (2,6) {$\union$};
  \node[draw,circle,fill=blue!10, inner sep=0] (h) at (2,7) {$\join$};
  \node[draw, circle, fill=blue!10, inner sep=0] (g) at (0.85,6) {$\union$};
  \node[draw,circle,fill=red,label=below:$0$] (f) at  (0.9,5) {};
  \node[draw,circle,fill=red, label=below:$\frac{1}{2}$] (a) at (0,5) {};
  \node[draw, circle, fill=black, label=below:$-2$] (b) at (-1,3) {};
  \path 
        (f) edge node[right]{}(g)
        (g) edge [dashed]node[left]{}(a)
        (h) edge node[right]{}(j)
        (h) edge node[left]{}(g)
        (h) edge node[left]{}(k)
        (j) edge node[right]{}(l)
        (j) edge node[below]{}(m);
\end{tikzpicture}
       \caption{}
       \label{fig10}
        \end{minipage}
 \end{figure}
%
The assignments are made
$$
d_k    \leftarrow   \frac{2}{3} \hspace*{1cm} d_l   \leftarrow   0 .
$$
leaving the $d_k$ and $d_l$ unchanged,
as shown in Figure \ref{fig12}.
Vertex $v_k$ is removed from the cotree and $v_l$ relocated
to the cotree's root, as shown in Figure \ref{fig13}.
%
  \begin{figure}[H]
         \begin{minipage}[c]{0.45 \linewidth}
\begin{tikzpicture}
  [scale=1,auto=left,every node/.style={circle,scale=0.9}]
  \node[draw,circle,fill=black,label=below:$0$] (p) at (-2,3) {};
  
  \node[draw,circle,fill=black,label=below:$-\frac{3}{2}$] (n) at (1,3) {};
   \node[draw,circle,fill=black,label=below:$-2$] (q) at (0,3) {};
  \node[draw, circle, fill=black, label=below:$\frac{2}{3}$] (m) at (2,5) {};
  \node[draw,circle,fill=black,label=below:$0$] (l) at (3,5) {};
  \node[draw,circle,fill=black,label=below:$0$] (k) at (3,6) {};
  \node[draw, circle, fill=blue!10, inner sep=0] (j) at (2,6) {$\union$};
  \node[draw,circle,fill=blue!10, inner sep=0] (h) at (2,7) {$\join$};
  \node[draw, circle, fill=black, label=below:$0$] (g) at (0.85,6) {};
  
  \node[draw,circle,fill=black, label=below:$\frac{1}{2}$] (a) at (2,3) {};
  \node[draw, circle, fill=black, label=below:$-2$] (b) at (-1,3) {};
  \path

        (h) edge node[right]{}(j)
        (h) edge node[left]{}(g)
        (h) edge node[left]{}(k)
        (j) edge node[right]{}(l)
        (j) edge node[below]{}(m);
\end{tikzpicture}
\caption{}
       \label{fig11}
       \end{minipage}\hfill
       \begin{minipage}[c]{0.45 \linewidth}
\begin{tikzpicture}
  [scale=1,auto=left,every node/.style={circle,scale=0.9}]
  \node[draw,circle,fill=black,label=below:$0$] (p) at (-2,3) {};
  
  \node[draw,circle,fill=black,label=below:$-\frac{3}{2}$] (n) at (1,3) {};
   \node[draw,circle,fill=black,label=below:$-2$] (q) at (0,3) {};
  \node[draw, circle, fill=red, label=below:$\frac{2}{3}$] (m) at (2,5) {};
  \node[draw,circle,fill=red,label=below:$0$] (l) at (3,5) {};
  \node[draw,circle,fill=black,label=below:$0$] (k) at (3,6) {};
  \node[draw, circle, fill=blue!10, inner sep=0] (j) at (2,6) {$\union$};
  \node[draw,circle,fill=blue!10, inner sep=0] (h) at (2,7) {$\join$};
  \node[draw, circle, fill=black, label=below:$0$] (g) at (0.85,6) {};
  
  \node[draw,circle,fill=black, label=below:$\frac{1}{2}$] (a) at (2,3) {};
  \node[draw, circle, fill=black, label=below:$-2$] (b) at (-1,3) {};
  \path

        (h) edge node[right]{}(j)
        (h) edge node[left]{}(g)
        (h) edge node[left]{}(k)
        (j) edge node[right]{}(l)
        (j) edge [dashed]node[below]{}(m);
\end{tikzpicture}
       \caption{}
       \label{fig12}
        \end{minipage}
  \end{figure}
%
In the penultimate step,
a depth one pair $\{v_k, v_l \}$ is chosen.
Since $ d_k= \alpha = 0$ and $d_l= \beta = 0$,
{\bf subcase~1a} is applied,
and assignments
$$
d_k    \leftarrow  -2 \hspace*{1cm} d_l  \leftarrow  \frac{1}{2}
$$
are made, as shown in Figure~\ref{fig14}.
%
  \begin{figure}[H]
         \begin{minipage}[c]{0.45 \linewidth}
\begin{tikzpicture}
  [scale=1,auto=left,every node/.style={circle,scale=0.9}]
  \node[draw,circle,fill=black,label=below:$0$] (p) at (-2,3) {};
  
  \node[draw,circle,fill=black,label=below:$-\frac{3}{2}$] (n) at (1,3) {};
   \node[draw,circle,fill=black,label=below:$-2$] (q) at (0,3) {};
  \node[draw, circle, fill=black, label=below:$\frac{2}{3}$] (m) at (3,3) {};
 
  \node[draw,circle,fill=black,label=below:$0$] (k) at (3,6) {};
  \node[draw, circle, fill=black, label=below:$0$] (j) at (2,6) {};
  \node[draw,circle,fill=blue!10, inner sep=0] (h) at (2,7) {$\join$};
  \node[draw, circle, fill=black, label=below:$0$] (g) at (0.85,6) {};
  
  \node[draw,circle,fill=black, label=below:$\frac{1}{2}$] (a) at (2,3) {};
  \node[draw, circle, fill=black, label=below:$-2$] (b) at (-1,3) {};
  \path         
        (h) edge node[right]{}(j)
        (h) edge node[left]{}(g)
        (h) edge node[left]{}(k);
\end{tikzpicture}
\caption{}
       \label{fig13}
       \end{minipage}\hfill
       \begin{minipage}[c]{0.45 \linewidth}
\begin{tikzpicture}
  [scale=1,auto=left,every node/.style={circle,scale=0.9}]
  \node[draw,circle,fill=black,label=below:$0$] (p) at (-2,3) {};
  
  \node[draw,circle,fill=black,label=below:$-\frac{3}{2}$] (n) at (1,3) {};
   \node[draw,circle,fill=black,label=below:$-2$] (q) at (0,3) {};
  \node[draw, circle, fill=black, label=below:$\frac{2}{3}$] (m) at (3,3) {};
 
  \node[draw,circle,fill=black,label=below:$0$] (k) at (3,6) {};
  \node[draw, circle, fill=red, label=below:$\frac{1}{2}$] (j) at (2,6) {};
  \node[draw,circle,fill=blue!10, inner sep=0] (h) at (2,7) {$\join$};
  \node[draw, circle, fill=red, label=below:$-2$] (g) at (0.85,6) {};
  
  \node[draw,circle,fill=black, label=below:$\frac{1}{2}$] (a) at (2,3) {};
  \node[draw, circle, fill=black, label=below:$-2$] (b) at (-1,3) {};
  \path        
        (h) edge node[right]{}(j)
        (h) edge [dashed]node[left]{}(g)
        (h) edge node[left]{}(k);
\end{tikzpicture}
       \caption{}
       \label{fig14}
        \end{minipage}
  \end{figure}
%
The final sibling pair
$\{v_k, v_l \}$ is chosen.
Since $ d_k= \alpha = \frac{1}{2} $ and $d_l= \beta = 0$, {\bf subcase~1a}
is applied and
$$
 d_k   \leftarrow  -\frac{3}{2} \hspace*{1cm} d_l   \leftarrow  \frac{2}{3}
$$
as shown in Figure \ref{fig15}.
When we remove $v_k$ using
Lemma~\ref{lem:cotreeremove}, the remaining cotree is $v_l$.
The algorithm stops, and the final diagonal
is shown in Figure \ref{fig16}.
 \begin{figure}[H]
        \begin{minipage}[c]{0.45 \linewidth}
\begin{tikzpicture}
  [scale=1,auto=left,every node/.style={circle,scale=0.9}]
  \node[draw,circle,fill=black,label=below:$0$] (p) at (-2,3) {};
  
  \node[draw,circle,fill=black,label=below:$-\frac{3}{2}$] (n) at (1,3) {};
   \node[draw,circle,fill=black,label=below:$-2$] (q) at (0,3) {};
  \node[draw, circle, fill=black, label=below:$\frac{2}{3}$] (m) at (3,3) {};
 
  \node[draw,circle,fill=red,label=below:$\frac{2}{3}$] (k) at (3,6) {};
  \node[draw, circle, fill=red, label=below:$-\frac{3}{2}$] (j) at (2,6) {};
  \node[draw,circle,fill=blue!10, inner sep=0] (h) at (2,7) {$\join$};
  \node[draw, circle, fill=black, label=below:$-2$] (g) at (4,3) {};
  
  \node[draw,circle,fill=black, label=below:$\frac{1}{2}$] (a) at (2,3) {};
  \node[draw, circle, fill=black, label=below:$-2$] (b) at (-1,3) {};
  \path

        (h) edge [dashed]node[right]{}(j)
        
        (h) edge node[left]{}(k);
\end{tikzpicture}
\caption{}
       \label{fig15}
       \end{minipage}\hfill
       \begin{minipage}[c]{0.45 \linewidth}
\begin{tikzpicture}
  [scale=1,auto=left,every node/.style={circle,scale=0.9}]
  \node[draw,circle,fill=black,label=below:$0$] (p) at (-2,3) {};
  
  \node[draw,circle,fill=black,label=below:$-\frac{3}{2}$] (n) at (1,3) {};
   \node[draw,circle,fill=black,label=below:$-2$] (q) at (0,3) {};
  \node[draw, circle, fill=black, label=below:$\frac{2}{3}$] (m) at (3,3) {};
 
  \node[draw,circle,fill=black,label=below:$\frac{2}{3}$] (k) at (3,7) {};
  \node[draw, circle, fill=black, label=below:$-\frac{3}{2}$] (j) at (5,3) {};
  \node[draw, circle, fill=black, label=below:$-2$] (g) at (4,3) {};
  
  \node[draw,circle,fill=black, label=below:$\frac{1}{2}$] (a) at (2,3) {};
  \node[draw, circle, fill=black, label=below:$-2$] (b) at (-1,3) {};
\end{tikzpicture}
       \caption{}
       \label{fig16}
        \end{minipage}
 \end{figure}
%
%
\section{Locating eigenvalues}
\label{sec:locating}
As an application, we can compute
in $O(n)$ time the number of eigenvalues of a cograph in a given interval,
as was done in \cite{JT2011} for trees and in \cite{JTT2013}
for threshold graphs.
The following theorem is called
Sylvester's Law of Inertia \cite[p. 336]{Bradley1975}.

\begin{Thr}
\label{thr:syl} Two $n \times n$ real symmetric matrices are congruent
if and only if they have the same number of positive eigenvalues
and the same number of negative eigenvalues.
\end{Thr}

The proof of the following theorem is similar to Theorem~3 in \cite{JTT2013},
and is based on Theorem~\ref{thr:syl}.
\begin{Thr}
\label{thr:mainA}
Let $D=[d_1, d_2, \ldots, d_n]$ be the diagonal returned by
\verb+Diagonalize+$(T_G , -a)$,
and assume $D$ has
$k_{+}$ positive values, $k_{0}$ zeros, and $k_{-}$ negative values.
\begin{my_description}
\item[{\em i}]
The number of eigenvalues of $G$ that are greater than $a$ is exactly $k_{+}$.
\item[{\em ii}]
The number of eigenvalues of $G$ that are less than $a$ is exactly $k_{-}$.
\item[{\em iii}]
The multiplicity of $a$ is $k_{0}$.
\end{my_description}
\end{Thr}

Note that when we use \verb+Diagonalize+$(T_G , x)$
and obtain $D = [d_1, d_2, \ldots, d_n]$, the order
of the $d_i$'s depends on the order in which maximum depth sibling pairs are selected.
It is not clear whether the values in $D$ are invariant, but their signs certainly are.

\begin{Ex}
\label{ex:two}
As Figure \ref{fig16} indicates,
\verb+Diagonalize+$(T_G , 0)$,
produces three positive entries, five negative entries and one zero.
Therefore $G$ has 3 positive eigenvalues, 5 negative eigenvalues and $0$ is an eigenvalue
with multiplicity one.
\end{Ex}

\begin{Ex}
\label{ex:three}
One can check that \verb+Diagonalize+$(T_G , 1)$,
produces a diagonal with multiset
$\{2, -\frac{1}{2}, 0, 0, 2, 2, -1, -\frac{1}{4}, 1 \}$.
Since there are four positive entries, three negative entries and two zeros,
it follows that there are four eigenvalues greater than -1,
three less than -1, and -1 has multiplicity two.
\end{Ex}

Assume $a < b$, and let $k_{+}$
be the number of positive values in the diagonal of \verb+Diagonalize+$(T_G , -a)$,
and let $j_+$ be number of positive values in the diagonal of
\verb+Diagonalize+$(T_G , -b)$.
Then $(a, b]$ must contain exactly $k_{+} - j_{+}$ eigenvalues.
The number of eigenvalues in $(a,b)$ is
$k_{+} - j_{+} -j_{0}$, where $j_{0}$ is the multiplicity of $b$.
Thus we can find the number of eigenvalues in an interval by making two
calls to the algorithm.

\begin{Ex}
From Example~\ref{ex:two} and Example~\ref{ex:three},
it follows that $G$ has exactly $4-3 = 1$ eigenvalue in $(-1, 0]$.
However, from Example~\ref{ex:two} we know that zero is an eigenvalue,
so there must be no eigenvalues in $(-1,0)$.
This is not an accident.
Recently in
\cite{Mohamadian2015},
the striking result was obtained that no cograph has eigenvalues in $(-1,0)$.
\end{Ex}

Any eigenvalue $\lambda$ can be approximated by first finding an interval for which
$$
a < \lambda \leq b  .
$$
By using a divide-and-conquer approach
in which the interval length is
successively cut in half,
one can find an interval $(a^\prime, b^\prime]$ of arbitrarily small size for which
$$
a^\prime < \lambda \leq b^\prime  .
$$
See \cite[Sec. 4]{JTT2013} for more detail.

\section{Inertia of cographs}
\label{sec:inertia}
The {\em inertia} of a graph $G$ is the triple
$(n_+, n_0, n_{-})$, where $n_+$, $n_0$,
$n_{-}$
denote respectively the number
of positive, zero, and negative eigenvalues of $G$.
We wish to compute the inertia of {\em any} cograph $G$.
If $G$ is a single vertex, its inertia is $(0,1,0)$,
so we assume that $G$ has order $n \geq 2$ with cotree $T_G$.
Note also that since $n = n_+ +   n_0  +  n_{-}$, it suffices to compute
only two of these components.
In what follows, we will obtain formulas for $n_{-}$ and
$n_0$ by using our diagonalization algorithm.
Our formula for $n_{-}$ given in Theorem~\ref{thr:inertia-neg} below appears to be new.
The formulas for $n_0$ (Theorem~\ref{thr:inertia-zer}) and the multiplicity of $-1$
(Theorem~\ref{thr:multminusone}) can be found in \cite{BSS2011}.
Formulas relating inertia to the representation
of threshold graphs can be found in the papers \cite{Bapat2013,JTT2015}.

By Theorem~\ref{thr:mainA},
we can obtain the inertia of a cograph by applying
\verb+Diagonalize+ to $T_G$ with $x = 0$,
and then counting the number of entries in the diagonal that are
positive, zero and negative.
The following technical result appears in Lemma~3 of \cite{Chang2008}.
\begin{Lem}
\label{lemmaC}
Suppose that  $ 0 \leq \alpha , \beta < 1.$ Then
\begin{enumerate}
\item[(a)] $0 < \frac{\alpha \beta -1}
{\alpha +\beta -2} < 1$
\item[(b)] $0 \leq \frac{\alpha \beta}{\alpha + \beta} < 1,$ provided $\alpha + \beta \neq 0.$
\end{enumerate}
\end{Lem}

\begin{Lem}
\label{lemmaB}
If $\{v_k, v_l\}$ is a sibling pair processed by \verb+Diagonalize+,
with parent $w=\join$ for which $0 \leq d_k, d_l < 1$,
then $d_k$ becomes permanently negative, and $d_l$ is assigned a value in $(0,1)$.
\end{Lem}
\begin{proof}
By our assumption, {\bf subcase~1a} is executed.  Hence
\begin{eqnarray*}
d_k & \leftarrow & \alpha + \beta  -2 \\
d_l & \leftarrow & \frac{\alpha \beta -1 }{\alpha + \beta -2}
\end{eqnarray*}
where $\alpha, \beta$ are the old values of $d_k, d_l$.
Clearly $d_k < 0$. By Lemma \ref{lemmaC}-(a), $d_l \in (0,1)$.
\end{proof}

\begin{Lem}
\label{lemmaA}
If $\{v_k, v_l\}$ is a sibling pair processed by \verb+Diagonalize+,
with parent $w=\union$ for which $0 \leq d_k, d_l < 1$,
then
$d_k$ becomes permanently nonnegative
and $d_l$ is assigned a value in $[0,1)$.
\end{Lem}
\begin{proof}
If $d_k  > 0$ or $d_l > 0$,
then
{\bf subcase~2a}
is executed which means
\begin{eqnarray*}
d_k \leftarrow \alpha + \beta \\
d_l \leftarrow \frac{\alpha \beta}{\alpha + \beta}
\end{eqnarray*}
Clearly $d_k > 0$ and by  Lemma \ref{lemmaC}-(b), $d_l \in [0,1)$.
If $d_k = d_l = 0$, then {\bf subcase~2b} is executed
meaning that both
$d_k$ and $d_l$ are assigned $0$,
$d_k$ permanently so.
\end{proof}

\begin{Lem}
\label{lem:range}
During the execution of
\verb+Diagonalize+$(T_G , 0)$,
all diagonal values of vertices remaining on the cotree are in $[0,1)$.
\end{Lem}
\begin{proof}
Initially all values on $T_G$ are zero.
Suppose after $m$ iterations of
\verb+Diagonalize+ all diagonal values
of the cotree are in $[0,1)$,
and consider iteration $m+1$ with sibling pair
$\{v_k, v_l\}$ and parent $w$.
By assumption, $0 \leq d_k, d_l < 1$.
If $w = \join$
then Lemma~\ref{lemmaB}
guarantees the vertex $d_l$
remaining on the tree is assigned a value in $(0,1)$.
If $w = \union$, Lemma~\ref{lemmaA} guarantees $d_l \in [0,1)$.
This means  after $m+1$ iterations
the cotree $T_G - v_k$
satisfies the desired property, completing the proof.
\end{proof}

\begin{Rem}
\label{rem:bottomup}
Observe that if $w$ is an interior node in $T_G$ having $t$ children,
as the algorithm progresses bottom up
through the rules of Lemma~\ref{lem:cotreeremove},
each interior child of $w$
eventually is replaced by a leaf.
Thus when $w$ is ready to be processed
it will have $t$ leaves as children.
To simplify our analysis,
without loss of generality we can assume that
all $t-1$ sibling pairs are processed consecutively.
\end{Rem}

The following theorem shows that the quantity $n_{-}(G)$ can be
computed in linear time from the cotree.

\begin{Thr}
\label{thr:inertia-neg}
Let $G$ be a cograph with cotree $T_G$ having $\join$-nodes $\{w_1, \ldots, w_j \}$,
and assume each $w_i$ has $t_i$ children in $T_G$.  Then
$$
n_{-}(G)  = \sum_{i=1}^{j} (t_i -1) .
$$
\end{Thr}
\begin{proof}
In executing \verb+Diagonalize+$(T_G , 0)$,
consider an interior node $w_i$
in $T_G$ of type $\join$ with $t_i$ children.
By Remark~\ref{rem:bottomup}, when it becomes eligible to be processed,
it will have $t_i$ leaves,
and the algorithm will process $t_i - 1$ sibling pairs.
By Lemma~\ref{lem:range} all diagonal values on the cotree remain in $[0,1)$.
By Lemma~\ref{lemmaB}
each of the $t_i - 1$ sibling pairs will produce
a permanent negative value before $w_i$ is removed.
This shows
$$
n_{-}(G)  \geq \sum_{i=1}^j (t_i -1) .
$$
However Lemma~\ref{lemmaA} shows that processing a sibling pair
with parent $\union$ can only produce nonnegative permanent values.
Hence the inequality is tight, completing the proof.
\end{proof}

We now consider $n_{0}(G)$.
A formula by B{\i}y{\i}ko{\u{g}}lu, Simi{\'c} and Stani{\'c} is known
\cite[Cor. 3.2]{BSS2011}
but we give an alternate proof using our algorithm.

\begin{Rem}
\label{rem:bottomupunion}
Consider an interior node $w$ of type $\union$ with $k = s + t$ children,
where $s$ children are interior (of type $\join$) and $t$ children are terminal.
In the execution of \verb+Diagonalize+$(T_G , 0)$,
from Lemmas~\ref{lemmaB}~and~\ref{lem:range}
each $\join$ node will become positive.
Thus when $w$ is processed, it will have
$s$ leaves with positive values and $t$ leaves with zero.
\end{Rem}

\begin{Thr}
\label{thr:inertia-zer}
Let $G$ be a cograph with cotree $T_G$ having $\union$-nodes $\{w_1, \ldots, w_m \}$,
where $w_i$ has $t_i$ terminal children.
If $G$ has $j \geq 0$ isolated vertices, then
$$
 n_{0}(G)  =  j + \sum_{i=1}^m (t_i -1) .
$$
\end{Thr}
\begin{proof}
Consider the execution of \verb+Diagonalize+$(T_G , 0)$,
and first consider the case when $j=0$.
Let $w_i$ be an interior node
of type $\union$ having $k_i$ children where $s_i$ children are interior nodes
and $t_i$ are terminal.
By Remark~\ref{rem:bottomupunion},
when $w_i$ is ready to be processed it
will have $s_i$ positive children and $t_i$ zeros.
From Lemma~\ref{lemmaA} we see that {\bf subcase~2a} will be executed
$s_i$ times and {\bf subcase~2b} $t_i - 1$ times.
Each execution of {\bf subcase~2b} produces a permanent zero on the diagonal,
and so $w_i$ contributes to $t_i - 1$ zeros.
This shows $n_{0}(G)  \geq  \sum_{i=1}^m (t_i -1)$.
To obtain equality, we note that
no zero can be created when processing a sibling pair whose parent is $\join$.
If $G$ has $j > 0$ isolates,
then the root of $T_G$ has type $\union$ with $j$ children as leaves.
We claim $j$ additional zeros are created.
Indeed, $j-1$ are created through
{\bf subcase~2b}.
The last iteration,
either {\bf subcase~2a} or {\bf subcase~2b},
creates an additional zero.
\end{proof}

It is possible to frame Theorem~\ref{thr:inertia-zer} in terms of duplicate vertices.
Let $\{V_i\}$ be a partition over
the set of all vertices that are in duplicate pairs, such that
each $V_i$ contains mutually pairwise duplicate vertices,
and for $i \not= j$, $v \in V_i$ and $w \in V_j$ imply $N(v) \not= N(w)$.
Then $|V_i| = t_i$.

We mention two other known theorems that can be obtained through our algorithm.
While we omit the details, their proofs involve algorithm analysis when $x = 1$.
One such result \cite{Mohamadian2015}
is that no cograph has an eigenvalue in $(-1,0)$.
Another \cite{BSS2011} involves the multiplicity of $-1$:

\begin{Thr}
\label{thr:multminusone}
Let $G$ be a cograph with cotree $T_G$ having $\join$-nodes $\{w_1, \ldots, w_m \}$,
where $w_i$ has $t_i$ terminal children.
Then the multiplicity of $-1$ is $\sum_{i=1}^m (t_i -1)$.
\end{Thr}

We apply the theorems in this section to the graph of Figure \ref{fig:cotree}.
By Theorem~\ref{thr:inertia-zer}, we have $n_{0}(G) =  2-1 = 1$.
Applying Theorem~\ref{thr:inertia-neg},
$n_{-}(G)  =  (3-1) + (3-1) + (2-1) = 5$.
Therefore $n_{+}(G) =   9 - 5 - 1 = 4$.
By Theorem~\ref{thr:multminusone}
the multiplicity of $-1$ in $G$ is two.
These numbers agree with those given in Section~\ref{sec:locating}.

\section{Equienergetic cographs}
\label{sec:energy}
Recall that the {\em energy} $E(G)$ of a graph $G$
is defined to be $\sum_{i=1}^n | \lambda_i |$,
where $\lambda_1, \ldots, \lambda_n$  are its eigenvalues.
If $E(G) = E(H)$, we say $G$ and $H$ are {\em equienergetic}.
It is known that the complete graph $K_n$ has energy $2(n-1)$.
We finish this paper by exhibiting an infinite class
$\Goneclass = \{ G_{1}, G_{2}, \ldots, \Gofr, \ldots  \}$
where $\Gofr$ is a cograph of order $n=3r+4$ and $E( \Gofr ) = E(K_{3r+4})$.
In \cite{JTT2015} the authors gave similar examples of threshold graphs,
however our present example involves non-threshold graphs.
This is of interest because equienergetic examples
seem to often involve non-integer energy.

The following two technical lemmas describe the diagonalization algorithm
when multiple leaves of the same parent have the {\em same} diagonal value $d_i = y$.
In particular, when an interior vertex $w$ of the cograph has $m$ terminal children,
$v_1, \ldots, v_m$, the algorithm will generally make $m-1$ iterations.
If all $d_i$ are equal, then under certain conditions,
can say exactly what assignments are made at each iteration.
We can assume sibling pairs are chosen to be the leftmost pair.
In Lemma~\ref{rem:join},
we prove the case when $w = \join$,
and the case $w = \union$
in Lemma~\ref{rem:union}
is handled in a similar way.

\begin{Lem}
\label{rem:join}
If $v_1, \ldots, v_m$ have parent $w=\join$,
each with diagonal value $y > 1$,
then the algorithm performs $m-1$ iterations of {\bf subcase 1a} assigning, during iteration $j$:
\begin{eqnarray}
d_{k} & \leftarrow & \frac{j+1}{j}(y -1)  \label{eq:remjoina}  \\
d_{l} & \leftarrow & \frac{y+j}{j+1}     \label{eq:remjoinb}
\end{eqnarray}
\end{Lem}
\begin{proof}
It is easy to check that when $j=1$,
since $\alpha = \beta = y > 1$,
{\bf subcase 1a} will be chosen and perform assignments
$d_{k} \leftarrow  2 (y -1)$
and
$d_{l} \leftarrow \frac{y+1}{2}$.
Now suppose iteration $j$ makes assignments
(\ref{eq:remjoina}) and (\ref{eq:remjoinb}).
Then during iteration $j+1$, we will have
$\alpha = \frac{y+j}{j+1} > 1$ and $\beta = y > 1$.
Therefore $\alpha + \beta \neq 2$ and {\bf subcase 1a} will execute again.
From the rules of the algorithm we get
\begin{eqnarray*}
d_{k} & \leftarrow & \alpha + \beta - 2 =  \frac{y+j}{j+1} + y -2 = \frac{j+2}{j+1}(y-1) \\
d_{l} & \leftarrow & \frac{\alpha \beta - 1}{\alpha + \beta -2} = (\frac{y^2 + jy}{j+1}  - 1)\frac{j+1}{(j+2)(y-1)} = \frac{j + y + 1}{j+2}
\end{eqnarray*}
which completes the induction.
\end{proof}

\begin{Lem}
\label{rem:union}
If $v_1, \ldots, v_m$ have parent $w=\union$,
each with diagonal value $y > 0$,
then the algorithm performs $m-1$ iterations of {\bf subcase 2a},
assigning during iteration $j$:
\begin{eqnarray}
d_{k} & \leftarrow & \frac{(j+1)y}{j} \label{eq:remuniona}  \\
d_{l} &\leftarrow &   \frac{y}{j+1} \label{eq:remunionb}
\end{eqnarray}
\end{Lem}
\begin{proof}
Similar to Lemma~\ref{rem:join}.
\end{proof}

For each integer $r\geq 1$, we now define the cograph $\Gofr$ to be the join of $K_{r+2}$
with the disjoint union of $r+1$ copies of $K_2$, that is
$$
\Gofr = ( \underbrace{K_2 \union K_2 \union \ldots \union K_2)}_{r+1} \join K_{r+2}.
$$
$\Gofr$ has order $n = 3r+4$ and its cotree is shown in Figure \ref{Figu2}.
Let $m(\lambda;G)$ denote the multiplicity of an eigenvalue in $G$.

\begin{figure}[b]
\begin{tikzpicture}
  [scale=1,auto=left,every node/.style={circle,scale=0.9}]
  \node[draw,circle,fill=blue!10, inner sep=0] (a) at (2,2) {$\join$};
  \node[draw, circle, fill=black,label=left:$\ldots$, label=right:$r+2$] (b) at (2,1) {};
  \node[draw, circle, fill=blue!10, inner sep=0] (f) at (3.5,1) {$\union$};
  \node[draw,circle,fill=black, label=left:$1$] (c) at (0.75,1) {};
  \node[draw,circle,fill=black] (d) at (3.75,-1) {};
  \node[draw,circle,fill=black] (e) at (3,-1) {};
  \node[draw,circle,fill=black] (i) at (4.25,-1) {};
  \node[draw,circle,fill=black] (j) at (5,-1) {};

  \node[draw,circle,fill=blue!10, inner sep=0, label=left:$1$] (g) at (3.25,0) {$\join$};
  \node[draw,circle,fill=blue!10, inner sep=0,label=left:$\ldots$, label=right:$r+1$] (h) at (4.5,0) {$\join$};
  \path (a) edge node[left]{} (b)
        (a) edge node[below]{} (c)
        (a) edge node[below]{} (f)
      (h) edge node[below]{} (i)
      (h) edge node[below]{}  (j)
       (g) edge node[below]{} (d)
      (g) edge node[below]{} (e)
        (f) edge node[left]{} (g)
        (f) edge node[left]{} (h);
\end{tikzpicture}
       \caption{The cotree $T_{\Gofr}$}
       \label{Figu2}
\end{figure}
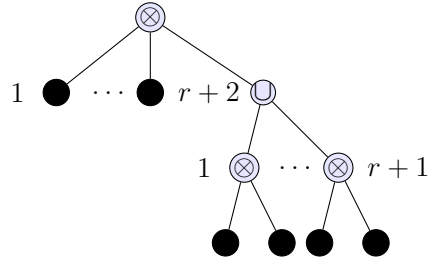

\begin{Lem}
\label{lema4}
$\Gofr$ has inertia
$(r+1, 0, 2r+3)$
and $m(-1; \Gofr) = 2(r+1)$.
\end{Lem}
\begin{proof}
This follows from Theorem~\ref{thr:inertia-neg}, Theorem~\ref{thr:inertia-zer} and Theorem~\ref{thr:multminusone}.
\end{proof}

\begin{Lem}
\label{lema7}
The eigenvalue $1$ in $\Gofr$ has multiplicity $r$.
\end{Lem}
\begin{proof}
By Theorem~\ref{thr:mainA} it suffices  to show that
\verb+Diagonalize+$(T_{\Gofr}, -1)$
creates exactly $r$ zeros.
After initializing vertices with $-1$,
the algorithm performs
{\bf subcase~1a}
on each of the $r+1$ sibling pairs
$\{v_k, v_l\}$ of depth three.
Each operation
\begin{eqnarray*}
d_{k} & \leftarrow & -4 \\
d_{l} &\leftarrow & 0
\end{eqnarray*}
leaves a zero on the tree at depth two.
After all depth three
pairs have been processed, the algorithm
applies {\bf subcase~2b}
to $r$ sibling pairs
$\{v_k, v_l\}$ where $d_k = d_l = 0$.
This creates $r$ permanent zeros,
and so $m(1; \Gofr) \geq r$.
By Lemma~\ref{lema4}, $\Gofr$ has exactly
$r+1$ positive eigenvalues.
However $m(1; \Gofr ) = r+1$
would contradict the well-known fact that
a graph's largest eigenvalue is simple.
So we have $m(1; \Gofr) = r$.
\end{proof}

Since $n = 3r + 4$,
Lemmas \ref{lema4} and \ref{lema7} account for all but two eigenvalues.

\begin{Lem}
\label{lema8}
The largest and smallest eigenvalues of $\Gofr$ are
\begin{eqnarray}
\lambda_1& = & 2r +3     \label{equ9} \\
\lambda_n & =& -(r+1) .   \label{equ10}
\end{eqnarray}
\end{Lem}
\begin{proof}
We will prove
(\ref{equ10}),
and
(\ref{equ9}) will follow since a graph's eigenvalues sum to zero.
To show (\ref{equ10}),
it suffices to prove that
\verb+Diagonalize+$(T_{\Gofr}, r+1)$
creates a non-negative diagonal with a single zero.
Consider the $r+1$ sibling pairs at depth three.
Since $\alpha = \beta = r+1$,
for each pair
{\bf subcase~1a} is executed.
The assignments
\begin{eqnarray*}
d_{k} & \leftarrow & 2r \\
d_{l} &\leftarrow & \frac{r+2}{2}
\end{eqnarray*}
are made, leaving $r+1$ vertices with a permanent positive value of $2r.$

At depth two, there are $r+1$ leaves all with positive value $y = \frac{r+2}{2}$,
and $r$ iterations are performed.
By Lemma~\ref{rem:union},
each iteration generates
the positive permanent diagonal
value of (\ref{eq:remuniona}).
On iteration $r$,
the assignment in (\ref{eq:remunionb})
$$
d_{l} \leftarrow   \frac{y}{j+1} = \frac{r+2}{2(r+1)}
$$
is made to a vertex $u$ which then gets moved under the root.

At depth one, there are $r+3$ leaves:
$r+2$ with diagonal value $r+1$,
and a single leaf $u$ whose diagonal value is $\frac{r+2}{2(r+1)}$.
Assume the algorithm processes the leaves with identical value first.
Letting $y = r+1$ in Lemma~\ref{rem:join},
we see that on each of the $r+1$ iterations
the algorithm generates the permanent positive diagonal value in (\ref{eq:remjoina}).
Letting $j = r+1$ in (\ref{eq:remjoinb})
we see that the last iteration leaves the value
$\frac{2(r+1)}{r+2}$ on the tree.
For the last step of the algorithm, we process the two remaining vertices
whose diagonal values are $\alpha = \frac{2(r+1)}{r+2}$
and $\beta = \frac{r+2}{2(r+1)}$.
Since $\alpha > 2$ and $\beta >0$,
on the last iteration {\bf subcase~1a} assigns
\begin{eqnarray*}
 d_k & \leftarrow &  \alpha + \beta -2  > 0 \\
 d_l  &\leftarrow & \frac{ \alpha \beta - 1}{\alpha + \beta -2} = 0
\end{eqnarray*}
creating a positive value and zero for the last two diagonal entries.
\end{proof}

From Lemma~\ref{lema4}, Lemma \ref{lema7}, and Lemma \ref{lema8} it follows that
$$
E(\Gofr ) = 2(r+1) + r + (2r +3) + (r+1) =  2(3r +4) - 2  = 2n -2
$$
so we have:
\begin{Thr}
\label{}
For each $r \geq 1$, $\Gofr$ and $K_{3r +4}$ are equienergetic.
\end{Thr}
Clearly $\Gofr$ and $K_n$ are noncospectral.
Note $\Gofr$ is not a threshold graph since its cotree is not a caterpillar.
Three other infinite classes of cographs, equienergetic to complete graphs, were discovered.
In all cases, the cotrees had depth three.
Since the cotree structure and proofs are similar to the above example, we omit them.

\bibliographystyle{amsplain}
\bibliography{paper}

\end{document}